\documentclass[nonblindrev,opre]{informs3} 
\usepackage{amsmath}
\usepackage{xcolor}
\usepackage{natbib}
\usepackage{graphicx}
\usepackage{amsfonts}
\usepackage{amssymb,verbatim,upref}
\usepackage{float,subcaption}
\usepackage{multirow}
\newtheorem{theorem}{Theorem}

\newtheorem{proposition}{Proposition}
\newtheorem{lemma}[proposition]{Lemma}
\newtheorem{example}[theorem]{Example}

\bibpunct[, ]{(}{)}{,}{a}{}{,}
\def\bibfont{\small}

\begin{document}
	\RUNAUTHOR{Anderson and Zachary}
	\RUNTITLE{Minimax decision rules}
\TITLE{Minimax decision rules for planning under uncertainty}
\ARTICLEAUTHORS{
\AUTHOR{Edward Anderson}
\AFF{Imperial College Business School, UK and}
\AFF{The University of Sydney Business School, Sydney, Australia, \EMAIL{edward.anderson@sydney.edu.au}}
\AUTHOR{Stan Zachary}
\AFF{Heriot-Watt University, UK and 
\AFF{University of Edinburgh, UK, \EMAIL{s.zachary@gmail.com}}}
}

\ABSTRACT{It is common to use minimax rules to make decisions for planning when there is great uncertainty on what will happen in the future. Minimax regret is one popular version of this. We give an analysis of the behaviour of minimax rules in the case with a finite set of possible future scenarios. The use of minimax rules avoids the need to determine probabilities for each scenario, which is an attractive feature in many public sector settings. However, minimax rules will have sensitivity to the choice of scenarios. In many cases using a minimax approach will mean the requirement for what may be regarded as arbitrary probabilities on scenarios is replaced by a similarly arbitrary choice of a very small number of specific scenarios. We investigate this phenomenon. When regret-based rules are used there are also problems arising since the independence of irrelevant alternatives property fails, which can lead to opportunities to game the process. Our analysis of these phenomena considers cases where the decision variables are chosen from a convex set in $R^n$, as well as cases with a finite set of decision choices.}

\maketitle

\section{Introduction}

We are interested in decision-making under uncertainty, and we are particularly concerned with the planning of infrastructure investment, e.g.\ in energy, where the value of such investments will depend on an unknown future. Many such projects will have lifetimes of decades and there are inevitably substantial uncertainties that will ultimately determine whether or not such an investment is worthwhile, and how it may compare with other possible choices of investment. Because of the long time horizons it is not easy to evaluate investments through a probabilistic analysis of how the future will evolve.

The difficulty of assigning probabilities to different possible future scenarios is central to our discussion. There are a number of situations that feature what is sometimes called \emph{deep uncertainty} \citep{Marchau2019}. These are cases where there is no agreement on fundamental aspects of what may happen in the future; often there are actions taken over time as the system evolves in unpredictable ways, and these actions are not under the influence of the decision maker. An example of this type of uncertainty in the UK at the present time relates to the future use of hydrogen as a fuel for domestic heating. One view is that a great deal of domestic heating will move to hydrogen, but another possibility is that there will be greater electrification through the use of heat pumps. What happens will depend on government actions and technology development as well as changing societal attitudes. If for planning purposes we have scenarios involving different proportions of domestic hydrogen heating being used in 2040 and beyond then there seems to be no reasonable way to assign probabilities to these different scenarios.

When faced with problems of this type planners will need to bear in mind that some decisions are self-reinforcing, so that a decision made now may make it more likely that a future outcome occurs which is favourable to that decision. Moreover, decisions are made over time and the way in which decisions either restrict or open up future options will be an important consideration. In this paper we will not address these types of concern. We assume that a decision is made at a single point in time, and that there is a well understood set of possible future scenarios, with descriptions that are not affected by the decisions made. Thus we consider a formal setup which can be described using a matrix of costs associated with different decisions and scenarios. This provides a useful starting point for any analysis even in cases where there is greater complexity.

Our focus is on minimax decision rules that are often used in these circumstances, and we will pay particular attention to minimax regret. Regret is an alternative form for the cost related to a particular scenario and decision pair: it is defined as the additional cost from the decision in comparison with the best possible decision for that scenario.

The approach of developing different possible future scenarios and then making decisions on a minimax regret basis is popular, and is sometimes called Least Worst Regret (LWR). For example, it has been proposed for decisions on facility location, waste management, and flood defences (\citet{Daskin97}, \citet{Chang2007}, \citet{VanderPol2017}). Planning decisions which relate to climate change and emissions have characteristics that often lead planners to use least worst regret as a decision tool (e.g. \citet{Loulou99}, \cite{Sanders2016}). There are two main factors that make minimax regret an attractive decision rule for planners.

First, a minimax approach is inherently conservative. Rather than looking at the average performance the decision maker's attention is focused on the worst outcome. We may regard this as an extreme form of risk aversion. In this context looking at regret rather than profit/cost makes sense since it then gives weight to scenarios in which the decision choice makes a significant difference. Suppose there is a single scenario in which some \textquotedblleft disaster\textquotedblright\ occurs and all decisions do badly, with higher costs than in any other scenario. Then a minimax cost policy will choose the decision which is least bad for this disaster scenario, even if the choice of policy makes almost no difference to the outcome in this scenario. It does not seem sensible to decide everything on the basis of a tiny improvement in a single scenario if this foregoes substantial improvement in most other scenarios, and this undesirable behaviour is avoided by using minimax regret.

Second, the ability to make decisions without the use of probabilities when there is deep uncertainty is often quoted as a reason for the use of minimax regret. In the public sector decisions may have significant financial consequences for individual firms and it is particularly important to make these in a way that follows a clear cut procedure -- to avoid any suggestion of bias or favouritism. The minimax regret approach achieves this without the need to assign probabilities to different possible outcomes.

There has been much discussion of the extent to which decision makers in practice may use estimates of future regret to determine their actions. Regret theory has been argued as being a natural way to explain some of the paradoxes that occur in actual decision making \citep{Bell82} where decision makers seem to have consistent preferences that cannot be easily explained under a fully rational model. The key point is that decision makers acting in an uncertain world not only consider the actual outcomes that can occur, but also what might have been possible. So the discovery that, as things have turned out, another alternative would have been preferable imparts a sense of loss or regret. Decision makers may be prepared to trade-off some amount of (expected) financial return to avoid this regret. There is debate about the extent to which these conclusions are supported by empirical evidence, and prospect theory as developed by Kahneman and Tversky is now widely used as a way of explaining how people choose between options when there is uncertainty on outcomes. But regret theory remains important and a summary of some more recent research in this area is given by \citet{Bleichrodt2015}. A minimax regret policy may match the way decisions are often made in practice when there is a shortage of information and analysis available. Thus it is arguable that minimax regret aligns with our usual mental heuristics, even if these are not entirely rational.

However since Savage's original paper on minimax regret \citep{Savage51} it has been recognised that a weakness of this approach is that it fails to satisfy a property of independence of irrelevant alternatives (IIA). This property implies that adding a single action (or changing the results of an action) cannot have an impact on the choice made by the decision maker unless the new or changed action could be chosen by the decision maker. It is also the case that a minimax regret approach may fail to satisfy transitivity, so that we can have a cycle of possible decisions where each is preferred to the next if just these two are compared on the basis of maximum regret.

The failure of the minimax regret rule to satisfy these very basic properties that we might expect in a good decision rule reflects more general results showing that decision makers who follow a set of reasonable axioms must act as though they are maximizing expected utility. More precisely we can find reasonable assumptions under which a rational decision maker will act as if there is a utility function $u$  and a set of \emph{subjective} probabilities $p_{i}$, $i\in S$, such that the decision maker chooses the decision $x$ that maximizes their expected utility $\sum_{i\in S}p_{i}u(C_{i}(x))$ where $C_i(x)$ is the cost occurring under scenario $i$ and decision $x$. The first person to provide a set of axioms under which this holds was Savage \citep{Savage54}. But there is a rich literature in this area: see for example \citet{Anscombe63} and \citet{Gul92}. These papers require a variety of different assumptions some of which are less likely to apply in the case that is our main focus, which is the planning of infrastructure investment. For examples Savage's original theorem requires the scenario set to be infinite. Many of these papers (though not the original work of Savage, or the paper by Gul) require it to be possible to define lotteries on actions (e.g.\ take action $x$ half the time and action $y$ half the time), so that the decision maker will have a preference between any pair of lotteries on actions. We may simplify the conclusions from this body of research by saying that any rule, such as minimax regret, that does not imply a probability measure on scenarios will lead to circumstances in which it fails to meet one or other axiom that one might regard as appropriate for rational decisions.

A number of authors have considered alternative sets of axioms for decisions that will imply a regret-based rule, or the minimax regret criterion. This work started with \citet{Milnor51} and includes \citet{Hayashi2008} and \citet{Stoye2011}. The choice of axioms vary between authors, but \citet{Hayashi2008} and \citet{Diecidue2017} discuss sets of axioms in which only regret-based decision rules are possible. Stoye gives a set of eight axioms, which if all satisfied, imply that the decision rule must be minimax regret.

However the nature of a regret-based rule implies that the choice made depends on the choice set available, even of options that are not chosen. This means that we can no longer insist on a transitivity condition, and the independence of irrelevant alternatives condition needs to be modified. Stoye's axioms that imply a minimax regret rule include two axioms related to the IIA and in both cases these are easily seen to hold for minimax regret: (axiom 4) when all choices have the same outcome in all scenarios, then introducing a new choice cannot change the decision, unless this new choice is selected; and (axiom 5) introducing a new choice with the property that in each scenario there is an existing choice that is at least as good cannot change the decision, unless this new choice is selected.

In this paper we will explore the implications of using minimax rules, and particularly the use of minimax regret. We highlight two key limitations that recur in different contexts. First the failure of the IIA condition is very widespread when regret functions are used, particularly with minimax regret. When a planner is choosing between projects from a number of firms this makes the decision open to being gamed. A second important limitation is that the number of scenarios required to determine the solution is no more that than $n+1$ in the case that there is a continuous decision problem in an $n$-dimensional space. This shows that the decision is critically dependent on the exact set of scenarios included in the analysis. We propose a modification that may help planners to improve their decision processes.

In the next section we set up the decision framework and show how this can extended both through the some variations in the definition of regret, and through redefining the uncertainty set in a robust interpretation of the minimax decision. In section 3 we discuss the behaviour of minimax rules when there are only a finite number of possible decisions. In section 4 we extend this to consider cases where the decision variables are chosen from some set in $\mathbb{R}^{n}$. In section 5 we look at a case where the decision is to choose which of a set of possible projects should be undertaken. In section 6 we look in more detail at a case where we balance capacity investment and risk. This gives an opportunity to use structural properties of the scenarios and we show how this works out in a case study from capacity procurement for electricity generation. 

\bigskip

\section{Framework}

We write $S$ for the set of possible scenarios, which we will assume is finite. We will assume a decision vector $x\in D\subseteq \mathbb{R}^{n}$ and for each $i\in S$ there is a cost function $C_{i}(x)$ giving the costs born by the decision maker if $x$ is chosen and then scenario $i$ occurs. Each cost function may represent a monetary cost, or some other measure. In the case of a single scenario $i$, the optimal decision $x$ would be that which minimised $C_{i}(x)$. But with multiple scenarios, the optimal decisions will in general differ between scenarios and so we aim to minimise some composite objective function $\bar{f}(x)$ defined on $D$.

There are many ways to define such a composite function. The classic expected utility framework uses $\bar{f}(x)=\sum_{i\in S}p_{i}f_{i}(x)$ where the $p_{i}$ are non-negative and each $f_{i}(x)$ is an (increasing) function of the costs $C_{i}(x)$. We can interpret $f_{i}(x)$ as the negative utility associated with costs $C_{i}(x)$, i.e. $f_{i}(x)=-u(C_{i}(x))$ and the $p_{i}$ (after normalising to sum to one) as probabilities associated with each scenario. Then minimizing $\bar{f}(x)$ is simply maximizing the expected utility for a utility function $u$. When $u$ is linear this becomes simply minimizing expected cost, and the risk averse case corresponds to concave $u$. Savage's Theorem implies that this form of function $\bar{f}(x)$ is the only possibility if we wish to satisfy some reasonable rationality assumptions. However there is still a difficulty in assigning appropriate probabilities to the scenarios.

A second possibility is to define 
\[
\bar{f}(x)=\max_{i\in S}C_{i}(x) 
\]
which leads to the minimax cost rule of choosing $x$ to solve $\min_{x\in D}\max_{i\in S}C_{i}(x)$ This is equivalent to the expected utility formulation above in the case that the utility function $u$ is sufficiently concave in costs, corresponding to extreme risk aversion. For example if the utility function is $u(z)=1-\exp(-kz)$ for payoff $z$, and we let $k \rightarrow \infty$ then for any fixed decision $x$, we have $f_{i}(x)=-1+\exp(kC_{i}(x))$ and the scenario $i^*$ with the largest cost $C_{i^*}(x)$ becomes dominant in the sum $\sum_{i\in S}p_{i}f_{i}(x)$ in any case where $p_{i^*}>0$, no matter how small this probability is. 

\subsection*{Regret-based decision rules}
A minimax decision rule does not make any distinction between scenarios, and does not make use of probability estimates for different scenarios. However, this type of extreme risk aversion may give rise to decisions that seem inappropriate. For example if there is a single scenario which has a high, but almost constant, cost across decisions then the outcomes in this single scenario will determine the choice made. 

The result of a minimax decision rule depends on the overall level of costs for each scenario. One difficulty in a planning context is that there will be costs that do not depend on the decision made, but do depend on the scenario that eventuates. In the Introduction we mentioned a scenario that could be related to the uptake of hydrogen for domestic heating in the UK. A true minimax cost policy for electricity generation capacity built would be critically influenced by the costs associated with building hydrogen infrastructure which are unrelated to generation capacity. It is often the case that these scenario based costs are hard to estimate. For this reason it is attractive to consider rules which depend only on the \emph{relative} costs in any given scenario. 

We consider decision rules based on regret. For each scenario $i\in S$, we define $R_{i}(x)$, the regret function defined on $D$, as 
\begin{equation}
\label{def:regret}
R_{i}(x)=C_{i}(x)-\inf_{z\in D}C_{i}(z). 
\end{equation}
Thus $R_{i}(x)$ represents the `regret' felt by the decision maker if she chooses $x$ and scenario $i$ occurs. Since, $\inf_{x\in D}R_{i}(x)=0$, $i\in S$, if in retrospect the best choice of $x$ was made for the scenario that occurs then there is no regret. Assuming that scenario $i$ occurs with probability $p_{i}$, the optimization facing the decision maker who wants to minimize expected regret is to find $\min_{x}\sum_{i\in S}p_{i}R_{i}(x)$, and this is equivalent to finding $\min_{x}\sum_{i\in S}p_{i}C_{i}(x)$ which is the expected cost. The minimax regret rule sets 
\[
\bar{f}(x)=\max_{i\in S}R_{i}(x) 
\]
so that the decision maker solves $\min_{x\in D}\max_{i\in S}R_{i}(x)$. This will produce a different result to the minimax cost rule.

We mentioned in the introduction the possibility of the minimax cost policy doing badly on some problem instances where one scenario gives high costs for all decisions, but with only a small difference between them. To support the idea that minimax regret may do better on average we compare the behaviour of the two decision rules on a random problem instance. We consider the simplest possible case with just two scenarios and three possible decisions $x,y,z$. Thus a problem instance is determined by the six cost values $C_i(x)$, $C_i(y)$ and $C_i(z)$, $i=1,2$. In a symmetric version of this problem we can suppose that each scenario is equally likely. Thus the expected cost for a given problem instance is the average cost over the two scenarios for the decision chosen, which can then be compared between the two decision rules. When the six cost values are all drawn randomly and independently from $[0,1]$ then the minimax cost rule gives an expected cost of $12/35=0.3429$. The minimax regret rule does better - though it is very hard to give an exact analysis, we can show by simulation that the expected cost is $0.3286$.  

In the case that the set $D$ of possible decisions is finite we may define a generalised regret through looking at some increasing symmetric function $r$ of the set of possible results obtained from different decisions. Then we define the generalised regret
\[
  R_{i}^{(r)}(x)=C_{i}(x)-r\{C_{i}(z):z\in D \}.
\]
Thus the standard definition of regret in \eqref{def:regret} is equivalent to choosing $r\{c_1,c_2, \ldots, c_m \}=\min \{c_1,c_2, \ldots, c_m \}$.

Amongst this class of regret measures it may be natural to consider the mean and median regret defined by taking $r$ as the mean or median of the set. This corresponds to a decision maker who is judged after the fact by looking at the performance of the chosen decision in comparison with the mean or median of all possible outcomes in the scenario that has occurred.

Much of our development will cover both cost and regret, and we write $f_{i}(x)$, $i\in S$, $x\in D$, for the general case where $f_{i}$ is either $C_{i}$ or $R_{i}^{(r)}$. 

\subsection*{Minimax as robust optimisation}

We will also generalise our discussion of minimax decision rules by considering them as robust versions of expected cost minimizers when there is a range of possible probabilities for the scenarios. We suppose that the set of allowed probability measures is restricted to a convex subset of the set of all probability measures, which we write as $\mathbb{P}_{A}$. This approach can be useful when a planner wishes to define some core scenarios around which a wider scenario set is constructed. For example these may represent different ways in which a government may choose to meet some defined policy goals in the future. This occurs for National Grid in their use of Future Energy Scenarios which reflect different ways in which a net zero GHG emissions target may be achieved in the UK by 2050 (see \cite{NationalGridFES}). \cite{Hughes2010} use the term \textquotedblleft backcasting\textquotedblright\  to refer to this way of generating scenarios. The report by \cite{Dent2020} commissioned by OFGEM in the UK also recommends an approach based around core scenarios.

The introduction of a subset $\mathbb{P}_{A}$ allows the definition of a partial ordering amongst the probabilities of different scenarios, $p_{i}\geq p_{j}$, etc. When there are core scenarios we can then define extreme scenarios, where each extreme scenario is guaranteed less likely than the corresponding core scenario. Or we may be more explicit and define a set of scenarios around each core scenario with probabilities defined conditional on one of these scenarios occurring (the core or one of its associated extremes). Within this framework we can also fix lower bounds on the probabilities of certain scenarios.

Given the set $\mathbb{P}_{A}$ of possible probability distributions over scenarios, we can formulate the robust optimisation problem: 
\[
\text{RO}(\mathbb{P}_{A}):~~\min_{x\in D}\max_{P\in \mathbb{P}_{A}}\mathbb{E}_{P}(f(x))
=\min_{x\in D}\max_{(p_{i}:i\in S)\in \mathbb{P}_{A}}\sum_{i\in S}p_{i}f_{i}(x). 
\]
In the case that $\mathbb{P}_A = \{p_{i},i\in S~|~p_{i}\geq 0,\sum p_{i}=1\}$ then RO$(\mathbb{P}_A )$ becomes the standard minimax problem, since
\[
\max_{(p_{i}:i\in S)\in \mathbb{P}_{A}}\sum_{i\in S}p_{i}f_{i}(x)=\max_{i\in S}f_{i}(x). \]

An important case occurs when $\mathbb{P}_{A}$ is defined by a finite set of linear constraints satisfied by the probabilities of each scenario. Thus we have a matrix $A$ of constraints with the probabilities $p_{i}$ of each scenario in the set 
$\mathbb{P}_{A}=\{p_{i},i\in S~|~Ap\leq 0,p_{i}\geq 0,\sum p_{i}=1\}$. 
This covers the case with no constraints when $A$ is empty, and the case with given probabilities through letting pairs of constraints such as $p_{1}-2p_{2}\leq 0$, $2p_{2}-p_{1}\leq 0$, define the ratios between probabilities and then using the fact that $\sum p_{i}=1$.

We will not make use of this in our discussion, but it is worth observing that the robust problem can be solved easily through dualising the inner maximization: 
\[
\begin{tabular}{ll}
maximize & $\sum_{i\in S}p_{i}f_{i}(x)$ \\ 
subject to & $Ap\leq 0,$ \\ 
& $\sum_{i\in S}p_{i}=1$, $p_{i}\geq 0.$
\end{tabular}
\]
We can formulate the dual to this linear program: Choose $w\in \mathbb{R}$, and $q$ in $\mathbb{R}^{m}$ to 
\[
\begin{tabular}{ll}
minimize$_{w,q}$ & $w$ \\ 
subject to & $w+\left( A^{\top }q\right) _{i}\geq f_{i}(x)$, $i\in S,$ \\ 
& $q\geq 0.$
\end{tabular}
\]
Then the original minimax problem can be written as the following optimization problem: Choose $x\in D,w\in \mathbb{R}$ and $q\in \mathbb{R}^{m}$ to 
\[
\begin{tabular}{ll}
minimize$_{x,w,q}$ & $w$ \\ 
subject to & $w+\left( A^{\top }q\right) _{i}\geq f_{i}(x)$, $i\in S,$ \\ 
& $q\geq 0.$
\end{tabular}
\]

There are three types of decision problem that we will investigate, depending on the set of possible decisions, $D$. First $D$ may be a finite set without any particular structure. Second $D$ may be a convex subset of $\mathbb{R}^{n}$, for example a planner may wish to determine the right capacities for some set of transmission links to be constructed. Third decisions may be required, simultaneously, on whether or not to go ahead with a number of different projects. This is a selection problem and we have $D=\{0,1\}^{n}$.\bigskip

\section{Minimax rules with finite decision sets}

In this section we use the simplest setting of $D$ finite to demonstrate the failure of some properties that we might expect. First we show that the Independence of Irrelevant Alternatives property fails to hold for the minimax regret problem. Thus removing an option which is not chosen can have an impact on the optimal choice. In fact, different choices for the characteristics of a new option that is added can lead to many different existing choices becoming optimal under the minimax regret rule. We show if a decision has the minimum cost under one of the scenarios, then it can become the minimax regret choice if a new option is added with correctly chosen costs.

\begin{lemma}
\label{Lem:Discrete_D_Gaming}Suppose that there is some $y\in D$ and $k\in S$, where $C_{k}(y)<C_{k}(x)$ for $x\neq y$. Then we may define a new possible decision $z$ and a set of values $C_{i}(z)$, $i\in S$ such that $y$ is chosen for the minimax regret problem $\min_{x\in D\cup \{z\}}\max_{i\in S}R_{i}(x)$.
\end{lemma}

\begin{proof}{Proof.} We write $M=\max_{j\in S}\max_{x}C_{j}(x)$ for the global maximum over the existing sets $D$ and $S$, and $L=\max_{j\in S}\min_{x\in D}C_{j}(x)$ for the largest of of the minimums over $D$ (which may be negative). We set $C_{i}(z)=M$ for $i\in S\backslash \{k\}$ and $C_{k}(z)=C_{k}(y)-M+L$. Then, since $M>L$, the decision $z$ replaces $y$ as the preferred choice in scenario $k$ with $R_{k}(y)=C_{k}(y)-C_{k}(z)=M-L$ and $R_{k}(x)=C_{k}(x)-C_{k}(z)>C_{k}(y)-C_{k}(z)=M-L$ for $x\neq y,z$. Thus $\max_{i\in S}R_{i}(x)>M-L$ for $x\neq y,z$. For $i\in S\backslash \{k\}$ we have $\min_{x}C_{i}(x)\leq L$, and so we can deduce that $R_{i}(z)=M-\min_{x}C_{i}(x)>M-L$ for $i\in S\backslash \{k\}$, so $\max_{i\in S}R_{i}(z)>M-L$. Also $R_{j}(y)=C_{j}(y)-\min_{x}C_{j}(x)<M-L$ for $j\neq k$, so $\max_{i\in S}R_{i}(y)=M-L$. This establishes the result. 
\halmos 
\end{proof}
\bigskip
It is easy to see that this result will also hold for the minimax mean regret problem, by scaling up the value of $M$ appearing in the proof. However Lemma \ref{Lem:Discrete_D_Gaming} does not hold for the minimax median regret problem, which is less influenced by extremely high or low costs. This has an advantage when there is a danger of the solution being gamed, i.e. manipulated through an artificial decision option being made available. 

We have a stronger result that the use of any function based on regrets to make decisions will fail the IIA property, unless the function reduces to simply minimizing expected cost under some  choice of probabilities over scenarios. Here we take the IIA property to mean that if decision $x$ is preferred to decision $x^{\prime }$, then this preference order is not changed by any change in the costs associated with a third decision $x^{\prime \prime } $. 
 
\begin{lemma}
\label{lem:rf} Suppose that, for a given decision problem, the decision space $D$ consists of at least three points, and the decision is made by minimizing a continuous non-decreasing function of the set of regrets  $f^{\ast }(R_{i}(x):i\in S)$ over $x\in D$ and in addition satisfies the IIA property, then the decision process is equivalent to minimizing the expected costs for some set of probabilities $p_{i}$, $i\in S$.
\end{lemma}

\begin{proof}{Proof.}  It will be convenient to assume $s$ scenarios and regard the set of regrets as a vector in $\mathbb{R}^{s}$, so $f^{\ast }:\mathbb{R}^{s}\rightarrow \mathbb{R}$. Moreover, since regret functions are non-negative, the function $f^{\ast }$ is naturally defined on $\mathbb{R}_{+}^{s}$. The IIA condition implies that, for any given scenario (assume without loss of  generality it is scenario~$1$) and for any two possible decisions $x$ and $x^{\prime }$ in $D$, we have that
\begin{multline}
f^{\ast }(R_{1}(x),...,R_{s}(x))>f^{\ast }(R_{1}(x^{\prime }),...,R_{s}(x^{\prime }))   \\
\implies f^{\ast }(R_{1}(x)+a,...,R_{s}(x))>f^{\ast }(R_{1}(x^{\prime})+a,...,R_{s}(x^{\prime }))  \label{implicationIIA}
\end{multline}
for all $a\geq -\min (R_{1}(x),R_{1}(x^{\prime })$. \ The reason for this is that the values $R_{1}(x)$ and $R_{1}(x^{\prime })$ of the regret $R_{1}$ may be changed by the same amount $a$ by  suitable adjustment of the cost $C_{1}(x^{\prime \prime })$ associated with one or  more of the remaining decisions $x^{\prime \prime }$; the IIA property implies that this should not change the direction of the inequality on the left hand side of \eqref{implicationIIA}. 

Iteration of \eqref{implicationIIA} implies that 
\begin{multline}
f^{\ast }(R_{1}(x),...,R_{s}(x))>f^{\ast }(R_{1}(x^{\prime }),...,R_{s}(x^{\prime }))   \\
\implies f^{\ast }(R_{1}(x)+a_{1},...,R_{s}(x)+a_{s})>f^{\ast }(R_{1}(x^{\prime })+a_{1},...,R_{s}(x^{\prime })+a_{s}),  \label{implicationIIA2}
\end{multline}
provided that all of $R_{1}(x)+a_{1},...,R_{s}(x)+a_{s}$ and $R_{1}(x^{\prime })+a_{1},...,R_{s}(x^{\prime })+a_{s}$ remain non-negative. Moreover we can easily establish the same implication with inequalities reversed. And the combination of both implications shows that
\begin{multline}
f^{\ast }(R_{1}(x),...,R_{s}(x))=f^{\ast }(R_{1}(x^{\prime }),...,R_{s}(x^{\prime }))   \\
\implies f^{\ast }(R_{1}(x)+a_{1},...,R_{s}(x)+a_{s})=f^{\ast}(R_{1}(x^{\prime })+a_{1},...,R_{s}(x^{\prime })+a_{s}). \label{implicationIIA3} 
\end{multline}

Now suppose that $f^{\ast }(\mathbf{R})=f^{\ast }(\mathbf{R}+\mathbf{a})$ for arbitrary vectors $\mathbf{R}$ and $\mathbf{a}$, with $\mathbf{R}\geq 0$ and $\mathbf{R}+\mathbf{a}\geq 0$. We will show that $f^{\ast }(\mathbf{R})=f^{\ast }(\mathbf{R}+q\mathbf{a})$ for any $q\in \mathbb{R}$. Since $f^{\ast }$ is continuous the set of points having the same value is closed, and hence it will be enough to show that $f^{\ast }(\mathbf{R})=f^{\ast }(\mathbf{R}+q\mathbf{a})$ for any rational number $q$ with $\mathbf{R}+q\mathbf{a}\geq 0$.

Consider an arbitrary rational $q=m_{1}/m_{2}$ with both $m_{1}$ and $m_{2}$ being integers. Then note that we have $f^{\ast }(\mathbf{R})=f^{\ast }(\mathbf{R}+(1/m_{2})\mathbf{a})$. To establish this we observe that if $f^{\ast }(\mathbf{R})>f^{\ast }(\mathbf{R}+(1/m_{2}) \mathbf{a})$ then from (\ref{implicationIIA2})  $f^{\ast }(\mathbf{R}+(1/m_{2})\mathbf{a})>f^{\ast }(\mathbf{R} +(2/m_{2})\mathbf{a})$ and we can continue in this way $m_{2}$ times to show $f^{\ast }(\mathbf{R})>f^{\ast }(\mathbf{R}+ \mathbf{a})$ which is a contradiction. In exactly the same way we get a contradiction if $f^{\ast }(\mathbf{R})<f^{\ast }(\mathbf{R}+(1/m_{2}) \mathbf{a})$. Thus we have established that  $f^{\ast }(\mathbf{R})=f^{\ast }(\mathbf{R}+(1/m_{2})\mathbf{a})$.  But now we may use (\ref{implicationIIA3}) repeatedly to show $f^{\ast }(\mathbf{R}+(1/m_{2})\mathbf{a})=f^{\ast }(\mathbf{R}+(2/m_{2})\mathbf{a})$, $f^{\ast }(\mathbf{R}+(2/m_{2})\mathbf{a})=f^{\ast }(\mathbf{R}+(3/m_{2})\mathbf{a})$ and so on (provided that we do not reach a vector with a negative component).  After $m_1$ repetitions we have the equality we want: $f^{\ast }(\mathbf{R})=f^{\ast }(\mathbf{R}+(m_{1}/m_{2})\mathbf{a})$.

Thus the set $L=\{x:$ $f^{\ast }(\mathbf{R})=f^{\ast }(\mathbf{R}+x)\}$ is the intersection of a linear subspace with $\mathbb{R}^{s}_{+}$. Continuity of $f^{\ast }$ implies that the set $\mathbf{R}+L$ divides $\mathbb{R}^{s}$ into two half spaces where  $f^{\ast }(\mathbf{R})<f^{\ast }(x)$ and $f^{\ast }(\mathbf{R})>f^{\ast}(x)$ and so $L$ has dimension $s-1$. We have $L$ defined by its normal vector $p$ with $L=\{x:p^{\top }x=0\}$. 

Now consider any other vector of regrets\ $\mathbf{R}^{\prime }$ with $f^{\ast }(\mathbf{R})\neq $ $f^{\ast }(\mathbf{R}^{\prime })$ Then we may use \eqref{implicationIIA3} with $\mathbf{a}=\mathbf{R}^{\prime }-\mathbf{R}$ to show that 
\[
f^{\ast }(\mathbf{R}^{\prime })=f^{\ast }(\mathbf{R}^{\prime }+x)\text{ for any }x\in L. 
\]
Thus $f^{\ast }$ is a function with sets of equal values all being translations of the linear subspace $L$. From this we deduce that we can write $f^{\ast }(x)$ as a function of the scalar $p^{\top }x$. 

We may scale so that $\sum p_{i}=1$, and then $f^{\ast }$ non-decreasing shows $p_{i}\geq 0$, so these can be interpreted as probabilities. We have established that $f^{\ast }(\mathbf{R})>f^{\ast }(\mathbf{R}^{\prime })$ if and only if  $\sum p_{i}R_{i}>\sum p_{i}R_{i}^{\prime }$ and the choice between two decisions is made by minimizing the expected regret under probabilities $p_{i}$.  As we observed earlier this is equivalent to minimizing expected costs, since
\[
\sum_{i\in S} p_{i}R_{i}(x)=\sum_{i\in S} p_{i}C_{i}(x)-\sum_{i\in S} p_{i}\inf_{z\in D}C_{i}(z).
\]
\halmos
\end{proof}

\bigskip 
The result of Lemma \ref{lem:rf} will also hold for some generalised regrets. The proof is unaltered except that we need to pay attention to the set $W_R$ of possible values that the generalised regret may take, since this is no longer guaranteed to be non-negative. We require that with three or more points in $D$ the generalised regret may take all values in $W_R$ even if the values of $C_i(x)$ for two points $x_1,x_2$ are given. This will be true for the mean regret and hence the lemma holds. However we cannot prove the same result for the median regret.

\cite{Shapiro2020} considered minimax decision rules and give conditions under which a solution to the minimax problem is also a solution to the problem of minimizing the expected costs for an appropriately chosen set of probability weights. Their result applies in the case that $D$ is a convex subset of $\mathbb{R}^{n}$ and the functions $f_{i}$ are convex with $S$ finite. We give an example to show that this result will not hold when there is a finite set of possible decisions and no specific structure for $f_{i}$. It is sometimes claimed that the fact that minimax gives equal treatment to all scenarios makes it likely to be similar to an expected cost approach with equal weights on all scenarios, but this example shows that it may not be possible to find even one probability distribution over scenarios that generates the same choice as a minimax rule.

\begin{example}
\label{Example:no_equiv_prob} In this example $D$ is the set $\{x,y,z\}$ and there are 3 scenarios. Costs are as follows (and these are also the regrets):
\[
\begin{tabular}{llll}
& $x$ & $y$ & $z$ \\ 
Scenario $A$ & $4$ & $0$ & $5$ \\ 
Scenario $B$ & $3$ & $5$ & $0$ \\ 
Scenario $C$ & $3$ & $2$ & $0$%
\end{tabular}
\]
It is easy to see that the minimax policy chooses $x$. Now suppose that we have probabilities $p_{A},p_{B}$ and $p_{C}$ for the three scenarios. Using $p_{A}=1-p_{B}-p_{C}$ the inequalities needed for $x$ to be selected under a minimum expected cost (or minimum expected regret) policy are
\begin{eqnarray*}
4-p_{B}-p_{C} &\leq &5p_{B}+2p_{C}, \\
4-p_{B}-p_{C} &\leq &5-5p_{B}-5p_{C},
\end{eqnarray*}
from which we deduce 
\begin{eqnarray*}
6p_{B}+3p_{C} &\geq &4, \\
4p_{B}+4p_{C} &\leq &1.
\end{eqnarray*}
These are incompatible with non-negative probabilities, since the second inequality implies that both $p_{B}$ and $p_{C}$ are less than $1/4$ which implies $6p_{B}+3p_{C}\leq 9/4$ contradicting the first inequality. So there is no assignment of probabilities to scenarios under which the minimax regret policy of $x$ will be chosen.
\end{example}

The failure of minimax regret to satisfy the IIA property also allows the possibility of non-transitivity in decision-making in the sense made clear in Example~\ref{ex:nt} below.  

\begin{example}
\label{ex:nt}
We take the decision space~$D$ to consist of three possibilities   $D = \{x,y,z\}$ and we again consider a scenario space~$S = \{A,B,C\}$ The costs of the decisions under each   scenario are given by
\[
    \begin{tabular}{rrrr}
                 &  $x$ &  $y$ &  $z$ \\
      Scenario A &  $4$ &  $0$ &  $2$ \\ 
      Scenario B &  $4$ &  $6$ &  $0$ \\ 
      Scenario C &  $0$ &  $0$ &  $5$ \\
    \end{tabular}
\]
It is now readily verified that, under the minimax regret criterion, if the decision~$z$ is not available then the decision~$y$ is preferred to the decision~$x$, while if the decision~$x$ is not available then the decision~$z$ is preferred to the decision~$y$, and, finally, if the decision~$y$ is not available then the decision~$x$ is preferred to the decision~$z$.
\end{example}

\section{Reduced scenarios for decisions in $\mathbb{R}^{n}$}

In the case where the decision set $D \subseteq \mathbb{R}^{n}$ is convex, we can establish that a reduced set of scenarios will determine the minimax regret policy (or indeed the minimax cost policy). For $K\subseteq S$, we write $x^{\ast }(K)$ for the solution to the minimax problem for the set of scenarios $K$, and cost or regret functions $f_{i}(x)$, thus 
\[
x^{\ast }(K)=\arg \min_{x}\max_{i\in K}f_{i}(x). 
\]
For the next result we will assume that the functions $f_{i}$ are quasiconvex, so that the level sets $A_{i}(z)=\{x \in D : f_{i}(x)\leq z\}$ 
are convex for each $z \in \mathbb{R}$.\bigskip

\begin{lemma}
\label{Lem:reduced scenarios A} Suppose that $D$ is a convex set in $\mathbb{R}^{n}$; for each $K\subseteq S$, the solution $x^{\ast }(K)$ to the minimax problem is uniquely defined; and for each $i\in S$, $f_{i}$ is quasiconvex. Then there is a set $K\subseteq S$, with $\left\vert K\right\vert \leq n+1$ and $x^{\ast }(K)=x^{\ast }(S)$.
\end{lemma}

\begin{proof} {Proof.}  Suppose that $R^{\ast }=\min_{x \in D}\max_{i\in S}f_{i}(x)$. Let $x_{0}=x^{\ast}(S)$ be the point achieving this minimum and set $H=\{i\in S:f_{i}(x_{0})=R^{\ast }\}$. From the definition and uniqueness of $x_{0}$ and by quasiconvexity, the level sets $A_{i}(R^{\ast })$, $i\in H$, are convex and $\bigcap_{i\in H}A_{i}(R^{\ast })=\{x_{0}\}$. Hence, for every choice of (direction)~$u$, there is some $i_{u}\in H$ such that $x_{0}+\delta u\notin A_{i_{u}}(R^{\ast })$ for all $\delta >0$, and further there is some hyperplane through~$x_{0}$ which separates $A_{i_{u}}(R^{\ast})$ and the set of points of the form $x_{0}+\delta u$ for $\delta >0$. This hyperplane may be defined by its normal~$g_{u}$, say. Then $g_{u}^{\top}u\geq 0$ and $g_{u}^{\top }(x-x_{0})<0$ for every $x\in A_{i_{u}}(R^{\ast}) $. Write $G=\{g_{u}:u\in \mathbb{R}^{n}\}$ and observe that $g_{u}^{\top}u\geq 0$ implies $0$ is in the convex hull of $G$. By Caratheodory's theorem we can choose a set of at most $n+1$ such hyperplanes $g_{u_{1}},\dots ,g_{u_{k}}$, $k\leq n+1$, with $0$ in the convex hull of $\{g_{u_{1}},\dots ,g_{u_{k}}\}$. This implies that, for any choice of direction~$u$, we have $g_{u_{j}}^{\top}u\geq 0$ for at least one $j=1,\dots k$. Let $K\subseteq H$ given by $K=\{i_{u_{1}},\dots , i_{u_{k}}\}$ be the corresponding set of associated
scenarios. Suppose that $x^{\ast }(K)\neq x_{0}$ and thus 
\[
\max_{i\in K}f_{i}(x^{\ast }(K))<R^{\ast }. 
\]
Since $f_{i}(x_{0})=R^{\ast }$ for all $i\in K$ we have $x^{\ast }(K)$ is in the interior of the level set $A_{i}(R^{\ast })$ for all $i\in K$ and thus $g_{u_{j}}^{\top }(x^{\ast }(K)-x_{0})<0$ for all $j=1,\dots k$ which is a contradiction, thus establishing $x^{\ast }(K)=x_{0}$ as required.
\halmos \end{proof}

\bigskip

We can extend this result to the robust optimisation problem  RO$(\mathbb{P}_A )$ in the important case that $\mathbb{P}_A$ is determined by a matrix $A$ which is block diagonal. Thus the scenario set $S$ may be partitioned into $m$ disjoint components $S=S_1 \cup S_{2}\cup \ldots \cup S_m$, and the matrix $A$ is such that each constraint (corresponding to some row of $A$) involves only scenarios within a single component of $S$.

We show below a version of Lemma \ref{Lem:reduced scenarios A} that establishes that the robust optimisation solution is determined by a set of scenarios in at most $n+1$ of the components $S_{k}$.

For $K\subset S$, we write $x_{A}^{\ast }(K)$ for the solution to the problem RO$(\mathbb{P}_A )$ restricted to the set of scenarios $K$, and cost or regret functions $f_{i}(x)$, thus 
\[
x_{A}^{\ast }(K)=\arg \min_{x \in D}\max_{(p_{i}:i\in S)\in \mathbb{P}_{A}}\sum_{i\in K}p_{i}f_{i}(x).
\]

\begin{lemma}
\label{Lem:reduced scenarios B}
Suppose that $D$ is a convex set in $\mathbb{R}^{n}$; $f_{i}$ is strictly convex for each $i\in S$; and $A$ is block diagonal inducing a split of $S$ into $m$ disconnected components $S_{1},S_{2},\ldots,S_{m}$. Then there is a set $L\subseteq \{1,2,\ldots, m\}$ with at most $n+1$ elements and $x_{A}^{\ast }(\cup _{j\in L}S_{j})=x_{A}^{\ast }(S)$.
\end{lemma}

\begin{proof}{Proof.}
Suppose that $R^{\ast }$ is the value of the problem RO$(\mathbb{P}_A )$, so 
\[
R^{\ast }=\min_{x\in D}\max_{(p_{i}:i\in S)\in \mathbb{P}_{A}}\sum_{i\in S}p_{i}f_{i}(x). 
\]
Let $x_{0}=x^{\ast }(S)$ be the point achieving this minimum. Thus $R^{\ast}$ is the optimal value for the inner maximization LP($x_{0}$), which can be written as the following linear program 
\[
\begin{tabular}{lll}
LP($x$): & maximize & $\sum_{i\in S}p_{i}f_{i}(x)$ \\ 
& subject to & $Ap\leq 0,$ \\ 
&  & $\sum_{i\in S}p_{i}=1$, $p_{i}\geq 0.$
\end{tabular}
\]
We break this up into the individual components. Let 
\[
\begin{tabular}{lll}
LP$_{j}$($x$): & maximize & $\sum_{i\in S_{j}}p_{i}f_{i}(x)$ \\ 
& subject to & $A_{_{j}}p\leq 0,$ \\ 
&  & $\sum_{i\in S_{j}}p_{i}=1$, $p_{i}\geq 0,$
\end{tabular}
\]
where $A_{j}$ is the appropriate component of $A$ involved in component $S_{j}$. We write $g_{j}(x)$ for the optimal value of LP$_{j}$($x$). So if $p_{i}^{(j)}(x)$ is an optimal solution to LP$_{j}$($x$) then 
\[
g_{j}(x)=\sum_{i\in S_{j}}p_{i}^{(j)}(x)f_{i}(x). 
\]
We note that $g_{j}$ is strictly convex, since if we take $x_{C}=\lambda x_{A}+(1-\lambda )x_{B}$ for $\lambda \in (0,1)$ then
\begin{eqnarray*}
g_{j}(x_{C}) &=&\sum_{i\in S_{j}}p_{i}^{(j)}(x_{C})f_{i}(x_{C}) \\
&<&\lambda \sum_{i\in S_{j}}p_{i}^{(j)}(x_{C})f_{i}(x_{A})+(1-\lambda )\sum_{i\in S_{j}}p_{i}^{(j)}(x_{C})f_{i}(x_{B}) \\
&\leq &\lambda g_{j}(x_{A})+(1-\lambda )g_{j}(x_{B})
\end{eqnarray*}
where the final inequality arises because $p_{i}^{(j)}(x_{C})$ are feasible for LP$_{j}$($x_{A}$) and LP$_{j}$($x_{B}$).

Then consider the following problem with decision variables $v_{j}$, $j=1,2,...,m$:
\[
\begin{tabular}{lll}
LP1($x$): & maximize & $\sum_{j=1}^{m}v_{j}g_{j}(x)$ \\ 
& subject to & $\sum_{j=1}^{m}v_{j}=1$ \\ 
&  & $v_{j}\geq 0$, $j=1,2,...,m.$
\end{tabular}
\]

From a feasible solution to LP1($x$) we can recover a feasible solution to LP($x$) by taking $p_{i}=v_{j}p_{i}^{(j)}$ for $i\in S_{j}$. This is feasible since the probability ordering $Ap\leq 0$ is retained and $\sum p_{i}=\sum_{j=1}^{m}v_{j}\left( \sum_{i\in S_{j}}p_{i}^{(j)}\right) =1$. Thus LP1($x$) is equivalent to LP($x$).

Now observe that LP1($x$) has a solution which is simply $\max \{g_{j}(x):j=1,2,...,m\}$ and, since strict convexity implies a unique solution to the minimax problem, we have converted the problem into one for which Lemma \ref{Lem:reduced scenarios A} applies (with $g_{j}(x)$ being quasiconvex, since it is convex). Thus there is a subset $L\subset \{1,2,...m\}$ with at most $n+1$ elements and 
\[
R^{\ast }=\min_{x}\max_{j\in L}g_{j}(x). 
\]
Since $g_{j}(x)=\sum_{i\in S_{j}}p_{i}^{(j)}(x)f_{i}(x)$ where $p_{i}^{(j)}(x)$ is an optimal solution to LP$_{j}$($x$), this establishes the desired result.
\halmos \end{proof}

\bigskip
In the case that there are a limited set of core scenarios each associated with a set of extreme scenarios, this result implies that only $n+1$ of the core scenarios will be used in determining the minimax decision. But the restrictions on probabilities within each component of $S$ will imply a relatively large number of individual scenarios contributing to the final result. We believe this can represent an attractive compromise in practice.

\section{Deciding amongst a set of projects}

An important area of application occurs when there are multiple projects that are available and the decision maker has to determine which projects to proceed with. We suppose that there is a set of projects $\cal{T}$ of size $n$. Thus we can write the decision space as $D=\{0,1\}^{n}$. In scenario $i\in S$ the selection $T\subset \cal{T}$ gives a cost to the decision maker of $C_{i}(T)$. Thus the regret function for a selection $T_{0}$ is $R_{i}(T_{0})=C_{i}(T_{0})-\inf_{T\subset \cal{T}}C_{i}(T)$ and the least worst regret choice is $\min_{T \subset \cal{T}}\max_{i\in S}R_{i}(T)$

The minimax decision rule can be used to determine which of a number of alternative projects should be carried out. For consistency we will consider this within a framework of costs, rather than considering net present values which might be more natural in this context.

An example of the use of minimax regret as a basis for project decisions of this sort occurs with the UK National Grid process for Network Options Assessment \citep{NOA2020}. Each year a process is carried out to determine which network reinforcement projects should be started. A minimax regret (or LWR) analysis is carried out as part of this decision process. 

A simple case of this sort occurs with \emph{independent additive costs}. Under this assumption the cost implications of different projects are additive where the individual project costs in scenario $i$ are $c_{i}(1),c_{i}(2),...,c_{i}(n)$. These are the extra costs associated with each project in scenario $i$ (and are negative if that project is profitable in scenario $i$). We also have a total cost in scenario $i$ of $W_{i}$ if none of the projects are selected. Hence 
\[
C_{i}(T)=\sum_{k\in T}c_{i}(k)+W_{i}. 
\]
Thus each scenario has a different additive cost function over $D$. Note that the values $W_{i}$ have no effect on the regret functions $R_{i}$, but will have an impact on a minimax cost rule. 

We note that the property of there being a reduced set of $n+1$ scenarios, that holds when $D$ is convex, does not apply in this discrete setting, even with independent additive costs. In the appendix we demonstrate this with an example where there are three projects, but 5 scenarios are all involved in the determination of the optimal project choice.

As the next two examples show, independent additive costs still allow poor behaviour by the minimax regret rule. 

\begin{example}
\label{Ex1}Suppose there are just two scenarios $A$ and $B$ and we consider whether or not to go ahead with two projects $X$ and $Y$. The costs involved under different scenarios, if the project $X$ goes ahead or not are given in the following table: 
\[
\begin{tabular}{ccc}
& Do not proceed with $X$ & Proceed with $X$ \\ 
Scenario $A$ & $\ 0$ & \ $3$ \\ 
Scenario $B$ & \ $4$ & \ $0$
\end{tabular}
\]
The numbers related to project $Y$ are the same. Costs are additive, so we have the following costs for the 4 options. 
\[
\begin{tabular}{ccccc}
& $\varnothing $ & $\{X\}$ & $\{Y\}$ & $\{X,Y\}$ \\ 
Scenario $A$ & $\ 0$ & $\ 3$ & $\ 3$ & $\ 6$ \\ 
Scenario $B$ & $8$ & $\ 4$ & $\ 4$ & $\ 0$
\end{tabular}
\]
The first column gives the values for $W_{A}$ and $W_{B}$. Then the minimax cost policy selects one of the projects $X$ or $Y$ for completion, but does not do both. This equally applies to minimax regret since regret and costs are the same for this example.

Hence we prefer both $X$ and $Y$ to doing nothing, yet we do not wish to do both projects. Costs are independent and additive but when the two projects are considered independently of each other, the decisions obtained are different from those which would be obtained when they are considered jointly.
\end{example}
\bigskip

In this case with $D=\{0,1\}^{n}$ the IIA property states that changes in the costs or availability of a particular selection cannot alter the preference between two other selections. For example a restriction that two projects interfere with each other and cannot both be carried out together will not effect the optimal solution unless the optimal solution has both these projects selected. However, there is a weaker \emph{project-based} IIA property that is also of interest. This can be defined as follows:  if a decision rule chooses from a set of projects $A$ a subset $S_{L}^{\ast }(A)$, the project-based IIA property is that $n\notin S_{L}^{\ast}(\{1,2,...n\})$ implies that $S_{L}^{\ast }(\{1,2,...n\})=S_{L}^{\ast }(\{1,2,...,n-1\})$. In other words leaving out of the consideration set a project which is not chosen makes no difference to the choice. The project-based IIA property follows from the usual IIA property simply by leaving out all the project selections involving the missing project one at a time. The project-based IIA  property will hold for minimax cost, but, as the next example demonstrates, it does not hold for minimax regret, even with independent and additive costs.

\begin{example}
\label{Ex2} There are 3 projects ($X,Y,Z$) and 3 scenarios. The cost values
are as follows
\[
\begin{tabular}{cccc}
& $X$ & $Y$ & $Z$ \\ 
Scenario A & $-1$ & $1$ & $0$ \\ 
Scenario B & $-1$ & $-2$ & $3$ \\ 
Scenario C & $1$ & $-2$ & $-2$
\end{tabular}
\]
In the case that $W_{A}=W_{B}=W_{C}=0$ we get cost values
\[
\begin{tabular}{ccccccccc}
& $\varnothing $ & $\{X\}$ & $\{Y\}$ & $\{Z\}$ & $\{X,Y\}$ & $\{X,Z\}$ & $%
\{Y,Z\}$ & $\{X,Y,Z\}$ \\ 
A & $0$ & $-1$ & $1$ & $0$ & $0$ & $-1$ & $1$ & $0$ \\ 
B & $0$ & $-1$ & $-2$ & $3$ & $-3$ & $2$ & $1$ & $0$ \\ 
C & $0$ & $1$ & $-2$ & $-2$ & $-1$ & $-1$ & $-4$ & $-3$
\end{tabular}
\]
\bigskip and regret values 
\[
\begin{tabular}{ccccccccc}
& $\varnothing $ & $\{X\}$ & $\{Y\}$ & $\{Z\}$ & $\{X,Y\}$ & $\{X,Z\}$ & $%
\{Y,Z\}$ & $\{X,Y,Z\}$ \\ 
A & $1$ & $0$ & $2$ & $1$ & $1$ & $0$ & $2$ & $1$ \\  
B & $3$ & $2$ & $1$ & $6$ & $0$ & $5$ & $4$ & $3$ \\ 
C & $4$ & $5$ & $2$ & $2$ & $3$ & $3$ & $0$ & $1$
\end{tabular}
\]
Hence the minimax regret policy chooses just project $Y$ (with a worst regret of $2$).
 
Now suppose that project $Z$ is not available. Then we obtain the following regret values:
\[
\begin{tabular}{ccccc}
& $\varnothing $ & $\{X\}$ & $\{Y\}$ & $\{X,Y\}$ \\ 
A & $1$ & $0$ & $2$ & $1$ \\ 
B & $3$ & $2$ & $1$ & $0$ \\ 
C & $2$ & $3$ & $0$ & $1$
\end{tabular}
\]
The minimax regret policy chooses both projects $X$ and $Y$ (with worst regret of $1$). Thus the presence or absence of project $Z$ will determine whether project $X$ is selected, even though $Z$ is not itself selected.
\end{example}

We can be more explicit in the case of minimax mean regret, or minimax median regret. If we look at the set of cost values in Example \ref{Ex2} for a single scenario, then we can see that the costs for different subset choices are symmetric about a mid point given by $(1/2)\sum_{k}c_{i}(k)+W_{i}$. This follows because the cost for a subset $T$ and the cost for the complement of $T$ are symmetric around this half-way point. Thus the mean regret and median regret are the same and given by
\[
  \bar{R}_{i}(T)=\sum_{k\in T}c_{i}(k)-(1/2)\sum_{k}c_{i}(k).
\]
In a formulation where we write $u_k=1$ for $k \in T$ and $u_k=0$ otherwise, we get $\bar{R}_{i}(u)=\sum_{k}c_{i}(k)(u_k - 0.5)$.  However, this formulation still does not imply that the project-based IIA property holds for minimax mean regret. We can consider the following example:
\[
\begin{tabular}{cccc}
& $X$ & $Y$ & $Z$ \\ 
Scenario A & $-2$ & $3$ & $4$ \\ 
Scenario B & $1$ & $-3$ & $2$ \\ 
Scenario C & $3$ & $-2$ & $-1$
\end{tabular}
\]
We will not work through the details, but it is easy to check that the minimax mean regret policy chooses to do none of the projects giving a maximum mean regret of 0. However, if the project $Z$ is left out then the optimal choice is to do both $X$ and $Y$ giving a maximum mean regret of 0.5.

\section{Balancing capacity investment and risk}
In a planning context it is often the case that there is a set of investments that need to be determined and there is uncertainty with respect to the demand that will occur. Investment gives greater capacity to meet high demands on the system. In these circumstances, rather than an individual scenario corresponding to a single set of demands, it may be better to let each scenario represent a distribution of these demands, but with a specific mean. The costs associated with the risk of demand exceeding capacity are then related to the safety margin between the set of capacities built and the average demand levels. 

In this case each scenario $i$ is associated with a vector of (mean) demands $a^{(i)}$. The decision to build a vector of capacities $x$ then has investment cost $k^{\top} x$ for some vector $k$ of costs, together with the costs associated with the risk of demand exceeding capacity in scenario $i$, which are often simply a function $h(x-a^{(i))}$ of the difference between mean demand and capacity.

Here there is a structure on the scenarios arising from the position of the mean demand vectors $a^{(i)}$ in $\mathbb{R}^n$, which is not something we have seen before. Our next result shows that the extreme scenarios (for the convex hull of the vectors $a^{(i)}$, $i \in S$) are the ones that determine the minimax solution. 

\begin{lemma}
\label{lem:ir}
Suppose that the costs in each scenario $i\in S$ are given by 
\begin{equation}
    \label{eq:simple}
    f_{i}(x)=h(x-a^{(i)})+k ^{\top} x,
\end{equation}
for some quasiconvex function $h$.  Suppose further that the scenario $j\in S$ is such that $a^{(j)}$ is a convex combination of the $a^{(i)}$ corresponding to the remaining scenarios, so that, for some $S_j\subset S$ with $j\notin S_j$, we have $a^{(j)}=\sum_{i \in S_{j}}\lambda_i a^{(i)}$ where each $\lambda_i\ge 0$ and $\sum_{i \in S_{j}}\lambda_i =1$.   Then the scenario $j$ can be removed from~$S$ without any effect on the solution of either the minimax cost problem or the minimax regret problem. 
\end{lemma}

\begin{proof}{Proof.}
We begin by showing that the regret functions are also of the form~\eqref{eq:simple}.  Let $x^*$  minimize  $h(x)+k ^{\top} x$. Then $x^*$ also minimizes $h(x)+k^{\top} (x+a^{(i)})$ and so $x^*+a^{(i)}$ minimizes $f_{i}(x)$. Thus the regret in each scenario $i$ is given by
\[
R_i(x)=h(x-a^{(i)})+k ^{\top} x -h(x^*)-k ^{\top} (x^*+a^{(i)})
\]
so that $R_i(x)=h(x-a^{(i)})+k ^{\top} (x -a^{(i)})$ plus a constant. Thus by redefining the function $h$ we obtain that $R_(x)$ is also of the form~\eqref{eq:simple} without any additional linear term (i.e.\ it is of the form $f_i(x)$ with $k=0$). 

Thus it is enough to show the result for the minimax cost problem. But now note that from $a^{(j)}=\sum_{i \in S_{j}}\lambda_i a^{(i)}$ for $j \notin S_{j} \subset S$, we may deduce
\begin{eqnarray*}
f_j(x) & = & h\biggl(\sum_{i \in S_{j}}\lambda_i (x - a^{(i)} )\biggr) + k ^{\top} x \\
            & \le &\max_{i \in S_{j}} h(x - a^{(i)}) + k^{\top} x \\
            & = & \max_{i \in S_{j}} f_i(x ) 
\end{eqnarray*}
where the inequality follows from the quasiconvexity of $h$.   Thus, removing scenario $j$ leaves the minimax problem unchanged as required.
\halmos
\end{proof}
\bigskip

Since, by Caratheodory's theorem, there is a set $K\subseteq S$ of size at most $n+1$ such that every $a^{(j)}$, $j\in S$, may be written as a convex combination of the $a^{(i)}$, $i\in K$, it follows that (analogously to Lemma~\ref{Lem:reduced scenarios A}), under the conditions of Lemma~\ref{lem:ir}, both the minimax cost and minimax regret decisions are determined solely by those scenarios in the set~$K$.

\subsection*{Example: electricity capacity procurement in Great Britain}
\label{sec:example}

Each year National Grid ESO (the GB electricity system operator) produces an Electricity Capacity Report (ECR), the purpose of which is to recommend a generation \emph{capacity-to-secure}, via a \textquotedblleft capacity auction\textquotedblright, in order to secure adequate GB electricity supplies in a given future winter. It is convenient for this example to take data from the 2015 ECR \citep{ECR2015} which was concerned with the provision of adequate electricity capacity for the winter 2019--20. The report considers a set~$S$ of 5 major, or core, scenarios for the above winter, together with a further 14 minor scenarios or \textquotedblleft sensitivities\textquotedblright. The core scenarios represent varying assumptions about the evolution of the electricity system and the evolution of demand, while the minor scenarios represent, for example, variations in the severity of the winter or in the level of renewable generation in that winter. The set~$D$ of possible decisions is essentially the set of capacities which might be recommended to be secured. For each scenario~$i$ in $S$ there is a cost function~$C_{i}$ on $D$ given by 
\begin{equation}
C_{i}(x)=\mathrm{VOLL}\times \mathrm{EEU}_{i}(x)+\mathrm{CONE}\times x,
\label{eq:11}
\end{equation}
where $x$ is the possible generation capacity to be secured, in MW, and $\mathrm{EEU}_{i}(x)$ is the corresponding \emph{expected energy unserved}, in MWh, over the given winter; the constant~$\mathrm{VOLL}$\ is the \emph{value of lost load}---taken in the above report as $\pounds 17,000$ /MWh---while the constant~$\mathrm{CONE}$\ is the \emph{cost of new entry} (procurement cost per unit of new generation capacity for the winter under study)---taken in the above report as $\pounds 49,000$/MW. The functions $\mathrm{EEU}_{i}(x)$ are estimated from data and are all approximately of the form $\exp (-\lambda (x - a^{(i)}))$ representing the exponential decay of risk as a function of the level of generation secured. 

Figure~\ref{fig:ecr_figure_14} is essentially a reproduction of Figure~14 of the 2015 ECR, and plots the cost functions~\eqref{eq:11} for the subset~$S^{\prime }$ of~$S$ consisting of the five major scenarios (plotted as solid lines) and two of the minor scenarios (plotted as dashed lines) considered in this report. The latter two scenarios are the most extreme of the set~$S$ of all the 19 scenarios considered, in the sense that their cost functions~\eqref{eq:11} bound (pointwise) above and below the cost functions for all the remaining scenarios. (As with Figure~14 of the 2015 ECR, Figure~\ref{fig:ecr_figure_14} omits some scenarios in order to avoid unnecessary clutter.)

\begin{figure}[!ht]
  \centering
  \includegraphics[scale=0.7]{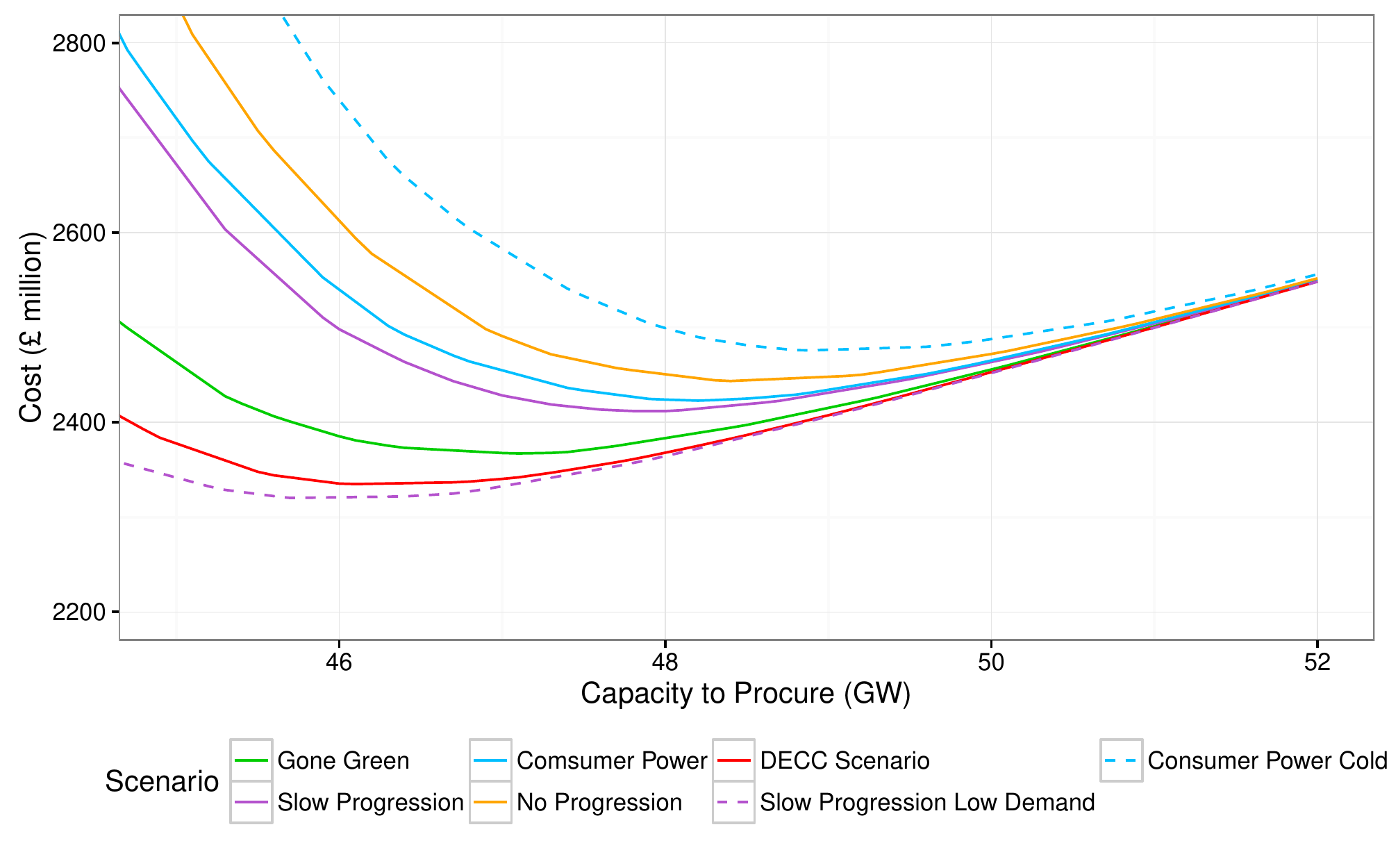}
  \caption{Cost functions $C_i(x)$ for each of the five major and two minor scenarios, as plotted in Figure 14 of the 2015 ECR.  The names attached to the scenarios are those used in the report.}
  \label{fig:ecr_figure_14}
\end{figure}

The 2015 ECR recommends a capacity-to-secure based on the use of minimax regret analysis. The regret functions $R_{i}(x)$ corresponding to the set~$S^{\prime }$ of the seven scenarios already considered above are plotted in Figure~\ref{fig:ECR_regret_functions}. It is easy to see that the set~$K$ consisting of the two minor scenarios plotted as dashed lines determine the result of the minimax regret analysis in the sense of Lemma~\ref{Lem:reduced scenarios A}, i.e. the minimax regret capacity-to-secure~$x^{\ast}(K)=47.8$ GW determined by just these two scenarios is the same as the minimax regret capacity-to-secure~$x^{\ast }(S^{\prime })$ determined by the 7 scenarios in the set~$S^{\prime }$. Indeed a minimax regret analysis of the entire set~$S$ of all 19 scenarios considered in the 2015 ECR shows that $x^{\ast }(K)=x^{\ast }(S)$, where $x^{\ast }(S)$ is the capacity-to-secure determined by the latter set---as would be evident if the regret functions corresponding to all 19 scenarios were plotted in Figure~\ref{fig:ECR_regret_functions}.

\begin{figure}[!ht]
  \centering
  \includegraphics[scale=0.7]{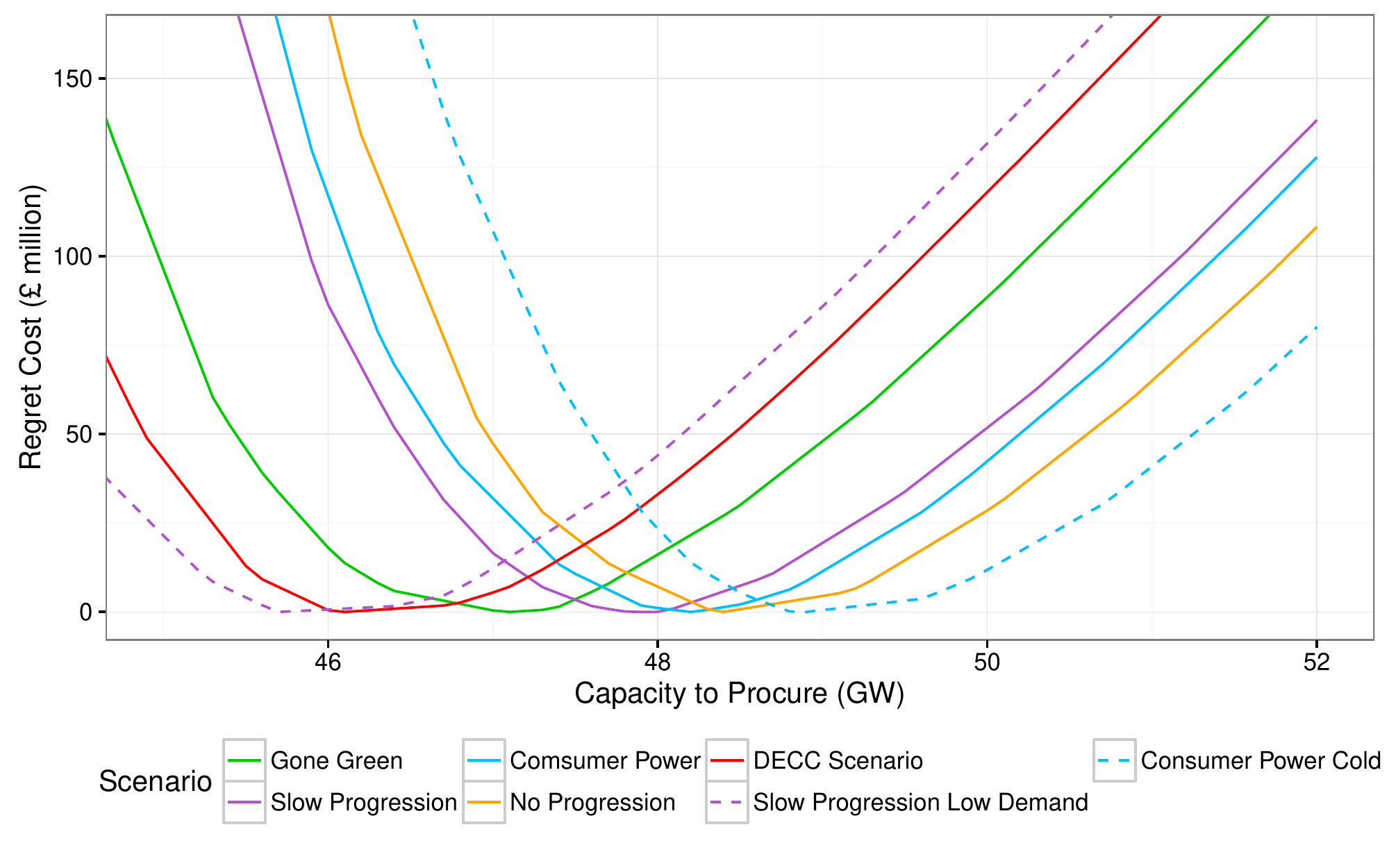}
  \caption{Regret functions~$R_i(x)$ for each of the five major and two minor scenarios illustrated in Figure 14 of the 2015 ECR.}
  \label{fig:ECR_regret_functions}
\end{figure}

It is thus the case that it is the two \textquotedblleft extreme\textquotedblright\ scenarios identified above which largely determine the result of the minimax regret analysis. The result of this analysis is unaffected by variation of the remaining scenarios so long as they do not become more extreme than either of these two scenarios. This is a consequence of Lemma~\ref{Lem:reduced scenarios A} which, as there is just a single decision variable, implies that the result is determined by just two scenarios. The analysis we give here (and presented to National Grid ESO) formed the basis for a thorough analysis of sensitivities in ECR 2017 \citep{ECR2017} (pages 87--97). The recommendation in that year proved robust to these further analyses, and a minimax regret analysis continues to be used up to the present. Each year the ECR is reviewed by a government appointed panel of technical experts and in their report for 2021 \citet{PTE-ECR2021} (page 27) concern is expressed over the extent to which the two scenarios determining the capacity requirement are extreme cases. One conclusion is that this methodology makes it important to take great care in the choice of scenario set. The problems that may occur in determining appropriate probabilities in a conventional analysis are to a great extent replaced by problems in determining the extreme scenarios that should be included.

\section{Conclusions}
\label{sec:conclusions}

An important property of the minimax decision-making criterion (whether applied to regret functions or to cost functions), is that the outcome of the analysis is typically determined by a small number of scenarios. This is an issue that needs to be considered by planners making use of these robust approaches, since sensitivity to the scenarios chosen reduces the apparent advantage (from a public planning perspective) of there being no requirement to assign probabilities to scenarios. We explore this issue and show how it can be helpful to view the minimax criterion as arising from a robust approach, with an uncertainty set which is some convex subset of the set of all probability measures. Then the use of partial ordering between the probabilities for ``core" and ``extreme" scenarios will widen the scenario set considered.

The paper has further focused on minimax regret analysis. It is well known that this fails to satisfy usually accepted conditions of economic rationality and in particular the IIA property. We give a variety of examples to show how widespread this behaviour is, both with continuous and discrete decision sets, whenever regret values are the basis for decisions. We show that with a continuous decision space the IIA property can only be retained when the regret values for different scenarios are combined through taking an expectation over a set of probabilities, which then reduces the problem  to one of minimizing expected costs.

In the case with a finite number of possible decisions we introduce a type of generalised minimax regret decision rule that defines regret in relation to some function of the set of outcomes associated with different decisions. We define the minimax median regret which may have some advantages in reducing the opportunities for gaming, and could be considered as an alternative decision rule for planners who are concerned about this possibility.  

One context in which minimax regret can be used as a planning tool occurs when projects are put forward for potential funding. The planner must then determined the selection of projects that will proceed. If the projects are proposed by a number of different firms, then the absence of the IIA property leaves minimax regret open to being gamed. Our analysis shows the importance of being careful in these circumstances.

\section*{Acknowledgements}

\label{sec:acknowledgements}

The authors would like to thank the Isaac Newton Institute for Mathematical Sciences for support during the Mathematics of Energy Systems programme 
\newline
(https://www.newton.ac.uk/event/mes/), when early work on this paper was undertaken.

\section*{Appendix: An example with $D=\{0,1\}^n$ and more than $n+1$ active scenarios}

We take $D=\{0,1\}^{n}$ and give an example where more than $n+1$ scenarios are active, in the sense that removing any one of them changes the minimax regret solution. We take 3 projects (X, Y and Z) and 5 scenarios (A, B, C, D and E) with the following set of costs
\[
\begin{tabular}{cccc}
costs & project X & project Y & project Z \\ 
scenario A & $6$ & $-2$ & $-4$ \\ 
scenario B & $2$ & $4$ & $4$ \\ 
scenario C & $4$ & $-8$ & $-1$ \\ 
scenario D & $-6$ & $6$ & $0$ \\ 
scenario E & $-2$ & $-7$ & $1$%
\end{tabular}
\]
The costs for different choices of sets of projects are:
\[
\begin{tabular}{ccccccccc}
& $\varnothing $ & \{X\} & \{Y\} & \{Z\} & \{X,Y\} & \{X,Z\} & \{Y,Z\} & \{X,Y,Z\} \\ 
scenario A & $0$ & $6$ & $-2$ & $-4$ & $4$ & $2$ & $-6$ & $0$ \\ 
scenario B & $0$ & $2$ & $4$ & $4$ & $6$ & $6$ & $8$ & $10$ \\ 
scenario C & $0$ & $4$ & $-8$ & $-1$ & $-4$ & $3$ & $-9$ & $-5$ \\ 
scenario D & $0$ & $-6$ & $6$ & $0$ & $0$ & $-6$ & $6$ & $0$ \\ 
scenario E & $0$ & $-2$ & $-7$ & $1$ & $-9$ & $-1$ & $-6$ & $-8$
\end{tabular}
\]
We can convert this into a set of regret figures:
\[
\begin{tabular}{ccccccccc}
& $\varnothing $ & \{X\} & \{Y\} & \{Z\} & \{X,Y\} & \{X,Z\} & \{Y,Z\} & \{X,Y,Z\} \\ 
scenario A & $6$ & $12$ & $4$ & $2$ & $10$ & $8$ & $0$ & $6$ \\ 
scenario B & $0$ & $2$ & $4$ & $4$ & $6$ & $6$ & $8$ & $10$ \\ 
scenario C & $9$ & $13$ & $1$ & $8$ & $5$ & $12$ & $0$ & $4$ \\ 
scenario D & $6$ & $0$ & $12$ & $6$ & $6$ & $0$ & $12$ & $6$ \\ 
scenario E & $9$ & $7$ & $2$ & $10$ & $0$ & $8$ & $3$ & $1$ \\ 
max regret & $9$ & $13$ & $12$ & $10$ & $10$ & $12$ & $12$ & $10$
\end{tabular}
\]
The minimax regret policy with all scenarios present is $\varnothing $, i.e. to accept none of the projects. But remove scenario A and we choose \{X,Y\}; remove scenario B and we choose \{X,Y,Z\}; remove scenario C and we choose \{X,Z\}; remove scenario D and we choose \{Y\}: and remove scenario E and we choose \{Z\}. Thus dropping any scenario will cause a change in the decision made, so that all 5 scenarios contribute to the decision.
\end{document}